\documentclass[12pt]{article}
\usepackage[utf8]{inputenc}\usepackage[english]{babel}
\usepackage[T1]{fontenc}\usepackage{amsmath}
\usepackage{amssymb}\usepackage{amsthm}
\usepackage{amsfonts}\usepackage{mathrsfs}
\usepackage{geometry}\usepackage{enumerate}

\DeclareMathOperator{\convw}{\xrightarrow[]{w}}
\DeclareMathOperator{\convws}{\xrightarrow[]{w^{\ast}}}

\newtheorem{theorem}{Theorem}
\newtheorem{definition}{Definition}
\newtheorem{example}{Example}
\newtheorem{remark}{Remark}

\newtheorem{proposition}{Proposition}
\newtheorem{lemma}{Lemma}
\newtheorem{corollary}{Corollary}
\makeatletter\renewcommand{\subsection}{\@startsection{subsection}{1}
{0pt}{3.25ex plus 1ex minus.2ex}{-1em}{\normalfont\normalsize\bf}}\makeatother

\begin{document}

\title{d-Operators in Banach Lattices}
\author{Eduard Emelyanov$^{1}$, Nazife Erkur\c{s}un-\"Ozcan$^{2}$, 
Svetlana Gorokhova$^{3}$\\ 
\small $1$ Sobolev Institute of Mathematics, Novosibirsk, Russia\\
\small $2$ Department of Mathematics, Faculty of Science, 
Hacettepe University, Ankara, Turkey\\
\small $3$ Uznyj Matematiceskij Institut VNC RAN, 
Vladikavkaz, Russia}
\maketitle

\begin{abstract}
We study operators carrying disjoint bounded subsets of a Banach lattice into 
compact, weakly compact, and limited subsets of a Banach space. Surprisingly, 
these operators behave differently with classical compact, weakly compact,
and limited operators. We introduce and investigate d-Andrews operators
which carry disjoint bounded sets onto Dunford--Pettis sets. \end{abstract}



\section{Introduction}

\hspace{5mm}
Disjoint bounded sequences play an important role in the Banach lattice theory.
In the present paper disjoint bounded sets are used for classifing operators possesing 
certain compactness properties between Banach lattices and Banach spaces. 
A disjoint set is a set whose elements are pair-wise disjoint.
We abbreviate disjoint bounded sets by \text{\rm d-bdd} sets.
The prefix \text{\rm d-} in the operator's name
means that each \text{\rm d-bdd} set is carried by the operator into a set
which is (weakly) compact, (order) bounded, limited, etc. 

Throughout this paper: vector spaces are real, operators are linear;
letters $X$ and $Y$ stand for Banach spaces; $E$ and $F$ for Banach lattices;
$B_X$ for the closed unit ball of $X$; $I_X$ for the identity operator on $X$;
$\text{\rm L}(X,Y)$ (resp., $\text{\rm K}(X,Y)$, $\text{\rm W}(X,Y)$) 
for the space of all bounded (resp., compact, weakly compact) 
operators from $X$ to $Y$; and $E_+$ for the positive cone of $E$.
We identify $X$ with its image $\widehat{X}$ in $X''$ 
under the canonical embedding $\hat{x}(f)=f(x)$.
The uniform convergence on a set $A$ of a sequence of functions  $(g_n)$ 
to some $g$ is denoted by $g_n\rightrightarrows g(A)$.

\begin{definition}\label{limited and DP set}
{\em
A bounded subset $A$ of $Y$ is:
\begin{enumerate}[a)]
\item 
(almost) {\em limited} if $f_n\rightrightarrows 0(A)$ for each (disjoint) 
\text{\rm w}$^\ast$-null sequence $(f_n)$ in $Y'$ (see \cite{BD}, \cite[Definition 2.3]{CCJ}).
\item 
an (almost) {\em Dunford--Pettis set} if $f_n\rightrightarrows 0(A)$ for each (disjoint)
\text{\rm w}-null $(f_n)$ in $Y'$ (see \cite[Theorem 1]{Andr}, \cite[p.228]{Bour}).
\end{enumerate}}
\end{definition}

\begin{definition}\label{limited and Andr operators}
{\em An operator $T: X\to Y$ is:
\begin{enumerate}[a)]
\item
(almost) {\em limited} (shortly, $T\in\text{\rm Lim}(X,Y)$, $T\in\text{\rm aLim}(X,Y)$) 
if $T(B_X)$ is an (almost) limited subset of $Y$ (equivalently, if $T$ s bounded and
$\|T'f_n\|\to 0$ for every (disjoint) \text{\rm w}$^\ast$-null sequence $(f_n)$ in $Y'$)  
(see \cite{BD,Elbour}).
\item
an (almost) {\em Andrews operator} (shortly, $T\in\text{\rm And}(X,Y)$, $T\in\text{\rm aAnd}(X,Y)$) 
if $T(B_X)$ is an (almost) Dunford--Pettis subset of $Y$ (equivalently, 
if $T$ s bounded and $\|T'f_n\|\to 0$ for every 
(disjoint) \text{\rm w}-null sequence $(f_n)$ in $Y'$).
\end{enumerate}
}
\end{definition}

\noindent
The ``almost versions'' in Definitions \ref{limited and DP set} and
\ref{limited and Andr operators} make sense when $Y$ is a Banach lattice.
The closed convex circled hull of each (almost) limited (resp., (almost) Dunford--Pettis) set
is an (almost) limited (resp., (almost) Dunford--Pettis) set. 
For every $A\subseteq X$, $A$ is a \text{\rm DP} subset of $X$ iff
$\widehat{A}$ is limited in $X''$ \cite[Theorem 1]{AEG}.
Since every compact set is limited and every limited set is a Dunford--Pettis set then 
$\text{\rm K}(X,Y)\subseteq\text{\rm Lim}(X,Y)\subseteq\text{\rm And}(X,Y)$ and
$\text{\rm Lim}(X,F)\subseteq\text{\rm aLim}(X,F)\subseteq\text{\rm aAnd}(X,F)$.

\begin{remark}\label{Andrews operators}
{\em
\begin{enumerate}[a)]
\item
Andrews's operator $T$ can be defined simply by the condition 
that $T$ s bounded and $T'$ is a Dunford--Pettis operator. 
\item
Let $Y$ be reflexive. Then: limited subsets agree with relatively compact ones \cite{BD}; 
\text{\rm w}$^\ast$- and \text{\rm w}-convergence agree in $Y'$;
and hence Dunford--Pettis subsets agree with relatively compact ones;
so, $\text{\rm K}(X,Y)=\text{\rm And}(X,Y)$ for every $X$.
\item
As $B_{c_0}$ is a Dunford--Pettis subset of $c_0$ yet it is neither limited nor weakly compact, 
$I_{c_0}\in\text{\rm And}(c_0)\setminus\bigl(\text{\rm Lim}(c_0)\bigcup\text{\rm W}(c_0)\bigl)$.
\end{enumerate}}
\end{remark}

\medskip
\noindent
We shall use the following fact \cite[Theorem 2.4]{AEG0}.

\begin{proposition}\label{enveloping norm}
{\em
Let $\text{\rm P}(E,F)$ be a subspace of $\text{\rm L}(E,F)$ closed in the operator norm.
Then $\text{\rm span}\bigl(\text{\rm P}_+(E,F)\bigl)$ is a Banach space under the enveloping norm:
$$
   \|T\|_{\text{\rm r-P}}=\inf\{\|S\|:\pm T\le S\in\text{\rm P}(E,F)\} \ \ \ \ 
   \ \ \bigl(T\in\text{\rm span}\bigl(\text{\rm P}_+(E,F)\bigl)\bigl).
   \eqno(1)
$$
}
\end{proposition}
\noindent

\bigskip
In Section 2 we introduce \text{\rm d}-(weakly)-compact,
\text{\rm d}-limited, and \text{\rm d}-Andrews operators
and investigate their basic properties.
Section 3 is devoted to completeness of the spaces of \text{\rm d}-(weakly) compact 
operators and to the domination problem for such operators.
In Section 4 we investigate \text{\rm d}-limited and \text{\rm d}-Andrews operators 
and discuss their ``almost versions''.

For further unexplained notation and terminology, we refer to \cite{AB1,Mey}.

\section{Definitions and examples}

\hspace{5mm}
In this section, we introduce \text{\rm d}-operators for (almost) order bounded,  
(weakly) compact, limited, and Dunford--Pettis subsets of the range. 
Examples justifying that these operators differ from the corresponding 
classical ones are presented. 

It is an open question whether each linear operator carrying 
\text{\rm d-bdd} sets onto bounded sets is bounded. Since each operator 
$T:E\to Y$ that carries order intervals to bounded sets is bounded (cf. \cite[Theorem 2.1]{Emel}),
for a positive answer to the question, it sufficies to prove that
each operator carring disjoint order bouned sequences to bounded ones is bounded.
It is worth noting that a linear operator 
from a Hilbert space to a normed space is bounded iff it 
bounded along orthonormal sequences \cite[Lemma 2]{Gor}.


\subsection{d-Order bounded operators.}
The following two classes of operators rely on order boundedness.

\begin{definition}\label{d-o-bdd}
{\em
An operator $T\in\text{\rm L}(E,F)$ is: 
\begin{enumerate}[$a)$]
\item
{\em\text{\rm d}-order bounded} ({\em $\sigma$\text{\rm d}-order bounded}),
whenever $T$ carries \text{\rm d-bdd} (countable) subsets of $E$ onto order
bounded subsets of $F$. The set of such operators denoted by $\text{\rm d-o}(E,F)$
(resp., by $\sigma\text{\rm d-o}(E,F)$).
\item
{\em\text{\rm d}-semi-compact} ({\em $\sigma$\text{\rm d}-semi-compact}),
whenever $T$ carries \text{\rm d-bdd} (countable) subsets of $E$ onto almost order
bounded subsets of $F$. The set of such operators is denoted by $\text{\rm d-semi}(E,F)$
(resp., by $\sigma\text{\rm d-semi}(E,F)$).
\end{enumerate}}
\end{definition}

\noindent
Clearly, $\text{\rm d-o}(E,F)\subseteq\text{\rm d-semi}(E,F)$ and
$\sigma\text{\rm d-o}(E,F)\subseteq\sigma\text{\rm d-semi}(E,F)$.
Let $\Gamma$ is an uncountable set. It is straightforward that
the identity operator $I_{\ell_\omega^\infty(\Gamma)}$ on 
the Banach lattice $\ell_\omega^\infty(\Gamma)$ of
all bounded countably supported functions from $\Gamma$ to $\mathbb{R}$ is a
$\sigma$\text{\rm d-o} yet not \text{\rm d}-semi-compact operator.
Note that $I_{\ell_\omega^\infty(\Gamma)}$ is a non-\text{\rm d-o} operator
whose bi-dual is \text{\rm d-o}. We do not know whether or not each 
\text{\rm d-o} operator is order bounded. 

\medskip
\noindent
If $E$ is lattice isomorphic to an \text{\rm AM}-space with a strong order unit
then $\text{\rm d-o}(E,F)=\text{\rm L}(E,F)$ for all $F$.
For a partial converse we have the following. Suppose that
$\text{\rm d-o}(E,F)=\text{\rm L}(E,F)$ for all $F$.
Then $I_E\in\text{\rm d-o}(E)$, and hence every \text{\rm d-bdd} 
subset of $E$ is order bounded. If, additionally $E$ is $\sigma$-Dedekind complete
then $E$ is lattice isomorphic to an \text{\rm AM}-space by \cite[Exercise 2, p.476]{AA}.
Furthermore, if the norm in $E$ is order continuous then $E$ is lattice isomorphic 
to $c_0(\Omega)$ for some set $\Omega$, and hence $\dim(E)<\infty$
because of $I_E\in\text{\rm d-o}(E)$. 

\subsection{d-Compact operators.}
Next, we pass to compactness in Banach spaces.

\begin{definition}\label{d-comp-def}
{\em 
An operator $T\in\text{\rm L}(E,Y)$ is {\em\text{\rm d}-compact} 
({\em $\sigma$\text{\rm d}-compact}) if $T(D)$ lies
in a compact subset of $Y$ for each \text{\rm d-bdd} 
(countable) subset $D$ of $E$. The set of such operators is denoted by 
$\text{\rm d-K}(E,Y)$ (resp., by $\sigma\text{\rm d-K}(E,Y)$).}
\end{definition}

\noindent
Clearly, $\text{\rm K}(E,Y)\subseteq\text{\rm d-K}(E,Y)$.
We do not know whether each linear operator carrying
\text{\rm d-bdd} sets onto relatvely compact sets is bounded.
The following example shows that a \text{\rm d}-compact operator need not be compact.

\begin{example}\label{example-d-c0}{\em
Let $1\le p<\infty$ and let $T:L^\infty[0,1]\to L^p[0,1]$ be 
the natural embedding. Since $\|T(r_n^+)-T(r_m^+)\|_p=2^{-\frac{2}{p}}$ 
for $n\ne m$, where $r_n$ is the 
n-th Rademacher function on $[0,1]$, the operator $T$ is not compact.
However, $T\in\text{\rm d-K}(L^\infty[0,1],L^p[0,1])$. Indeed, if 
$(f_n)$ is a sequence in a $\text{\rm d-bdd}$ subset $D$ of $L^\infty[0,1]$,
then either some function $f_{n_0}$ occurs infinitely many times in $(f_n)$ 
or $\|f_n\|_p\to 0$. In each of the cases, $(Tf_n)=(f_n)$ has 
a $\|\cdot\|_p$-convergent subsequence.}
\end{example}

\noindent
As every infinite-dimensional normed lattice contains a normalized disjoint sequence,
$\text{\rm d-K}(E)=\text{\rm L}(E)$ iff $\dim(E)<\infty$.

\begin{proposition}\label{d-K-elem-1}
{\em
\begin{enumerate}[$i)$]
Let $T\in\text{\rm L}(E,Y)$. The following conditions are equivalent.
\item
$T\in\text{\rm d-K}(E,Y)$.
\item
$T(D)$ is relatively compact for every disjoint normalized $D\subseteq E_+$.
\item
$T\in\sigma\text{\rm d-K}(E,Y)$.
\item
$T(D)$ is relatively compact for every disjoint normalized countable $D\subseteq E_+$.
\end{enumerate}
}
\end{proposition}

\begin{proof}
Implications $i)\Longrightarrow ii)$ and $i)\Longrightarrow iii)\Longrightarrow iv)$ are trivial.
It suffices to prove implications $ii)\Longrightarrow i)$ and $iv)\Longrightarrow ii)$.

$ii)\Longrightarrow i)$\
Let $D$ be a disjoint subset of $B_E$. Denote
$$
   D_1=\{x_+: x\in D\ \& \ x_+\ne 0\} \ \ \ \text{\rm and} \ \ \ 
   D_2=\{x_-: x\in D\ \& \ x_-\ne 0\}.
$$
If one of these sets is empty, say $D_2=\emptyset$, then
$$
   T(D)\subseteq T(D_1)\cup\{0\}\subseteq
   \text{\rm cch}\bigl(\bigl\{T(\|x\|^{-1}x): x\in D_1\bigl\}\bigl),
   \eqno(2)
$$
where $\text{\rm cch}(S)$ denoted the convex circled hull of $S$.
By the condition $ii)$, the set $\{T(\|x\|^{-1}x): x\in D_1\bigl\}$
is relatively compact. Then $T(D)$ is relatively compact due to (2).
If $D_1$ and $D_2$ are both nonempty then $D\subseteq D_1-D_2$, and hence
$$
   T(D)\subseteq\text{\rm cch}\bigl(\bigl\{T(\|x\|^{-1}x): x\in D_1\bigl\}\bigl)+
   \text{\rm cch}\bigl(\bigl\{T(\|x\|^{-1}x): x\in D_2\bigl\}\bigl).
   \eqno(3)
$$
It follows from (3) that $T(D)$ is relatively compact.

$iv)\Longrightarrow ii)$\
Let $(y_n)$ be a sequence in $T(D)$, where $D$ is a disjoint normalized 
subset of $E_+$.
If $y_n=0$ for infinitely many $n$-th then $(y_n)$ has a norm-null subsequence.
If $y_n=0$ only for finitely many $n$-th, by removing them, we may suppose 
all $y_n\ne 0$. For each $n$, pick $x_n\in D$ 
with $y_n=Tx_n$ and denote $D_0=\bigl\{x_n: n\in\mathbb{N}\bigl\}$. 
Then $D_0$ is a disjoint normalized countable (may be even finite) subset of $E_+$.
By the assumption $iv)$, $T(D_0)$ is relatively compact. Since the sequence 
$(y_n)=(Tx_n)$ lies in $T(D_0)$ then $(y_n)$ contains a norm-convergent subsequence.
\end{proof}

It follows from Proposition \ref{d-K-elem-1} $iv)$ that each M-weakly compact operator is {\em\text{\rm d}-compact}. On the other hand a  {\em\text{\rm d}-compact} operator need not to be M-weakly compact, (see \cite[Example on page 322]{AB1}).

\subsection{d-Weakly compact operators.}
Now, we consider weak compactness.

\begin{definition}\label{d-w-comp-def}
{\em 
An operator $T\in\text{\rm L}(E,Y)$ is 
\text{\rm d}-{\em weakly} ($\sigma$\text{\rm d}-{\em weakly}) {\em compact}
if $T(D)$ lies in a weakly compact subset of $Y$ for each \text{\rm d-bdd} 
(countable) subset $D$ of $E$. The set of such operators is denoted by $\text{\rm d-W}(E,Y)$
(resp., by $\sigma\text{\rm d-W}(E,Y)$).}
\end{definition}

\noindent
Clearly, $\text{\rm W}(E,Y)\subseteq\text{\rm d-W}(E,Y)$.
The following example shows that, in general, a \text{\rm d}-weakly compact
operator is neither weakly compact nor \text{\rm d}-compact.

\begin{example}\label{example-d-wc0}{\em
Consider an operator $T:C[0,1]\to c$ defined by 
$Tx=\bigl(x\bigl(\frac{1}{k}\bigl)\bigl)_{k=1}^\infty$.

By the Eberlein--Smulian theorem, for \text{\rm d}-weakly compactness of $T$,
we must show that each sequence in the image under $T$ of an \text{\rm d-bdd} set 
has a \text{\rm w}-convergent subsequence. So, let $(x_n)$ be a sequence 
in a \text{\rm d-bdd} subset $D$ of $C[0,1]$.
If $(x_n)$ has finitely many distinct elements then there is nothing to prove. Otherwise,
$(x_n)$ has a disjoint subsequence $(x_{n_i})$. Since the dual
$(C[0,1])'$ is a KB-space then $x_{n_i}\convw 0$, and hence  
$Tx_{n_i}\convw 0$, as desired. Thus, $T\in\text{\rm d-W}(C[0,1],c)$.

For each $n$, define $g_n\in C[0,1]$ by: 
$g_n(t)=0$ for $t\le\frac{1}{n+1}$, $g_n(t)=1$ for $t\ge\frac{1}{n}$,
and linear otherwise. Since the sequence $(Tg_n)=\bigl(\sum_{i=1}^n e_i\bigl)$,
where $e_n$ is an n-th unit vector of $c$, has no \text{\rm w}-convergent 
subsequence, we conclude $T\notin\text{\rm W}(C[0,1],c)$. 

Finally, we show that $T$ is not \text{\rm d}-compact. Let 
$x_n(t)=0$ for $\bigl|t-\frac{1}{n+1}\bigl|>\frac{1}{4(n+1)^2}$, $x_n\bigl(\frac{1}{n+1}\bigl)=1$, 
and linear otherwise. Clearly, $\{x_n\}_{n=1}^\infty$ is a \text{\rm d-bdd} subset of $C[0,1]$, 
and $Tx_n=e_{n+1}\in c$ for each $n\in\mathbb{N}$.
Since the set $T(\{x_n\}_{n=1}^\infty)=\{e_k: k\in\mathbb{N}\}$ is 
not relatively compact in $c$ then $T\notin\text{\rm d-K}(C[0,1],c)$.}
\end{example}

The following fact is similar to Proposition \ref{d-K-elem-1}.

\begin{proposition}\label{d-W-elem-1}
{\em
\begin{enumerate}[$i)$]
Let $T\in\text{\rm L}(E,Y)$. The following conditions are equivalent.
\item
$T\in\text{\rm d-W}(E,Y)$.
\item
$T(D)$ is relatively \text{\rm w}-compact for every disjoint normalized $D\subseteq E_+$.
\item
$T\in\sigma\text{\rm d-W}(E,Y)$.
\item
$T(D)$ is relatively \text{\rm w}-compact for every disjoint normalized countable $D\subseteq E_+$.
\end{enumerate}
}
\end{proposition}

\begin{proof}
Implications $i)\Longrightarrow ii)$ and $i)\Longrightarrow iii)\Longrightarrow iv)$ are trivial.
The proof of $ii)\Longrightarrow i)$ is similar to the proof of $ii)\Longrightarrow i)$
in Proposition \ref{d-K-elem-1} via using of the Mazur theorem (cf., \cite[Theorem 3.4]{AB1})
for relatively \text{\rm w}-compactness of the right parts of formulas (2) and (3).

$iv)\Longrightarrow ii)$\
Let $D$ be a disjoint normalized subset of $E_+$. 
In view of the Eberlein--Smulian theorem, for proving that $T(D)$ is relatively \text{\rm w}-compact
it is enough to show that each sequence in $T(D)$ has a \text{\rm w}-convergent subsequence.
So, let $(y_n)$ be a sequence in $T(D)$.
As in the proof of Proposition \ref{d-K-elem-1}, we suppose all $y_n\ne 0$,
pick $x_n\in D$ with $y_n=Tx_n$, and denote $D_0=\bigl\{x_n: n\in\mathbb{N}\bigl\}$.
Then $T(D_0)$ is relatively \text{\rm w}-compact by $iv)$.
Since $(y_n)$ lies in $T(D_0)$ then (by the Eberlein--Smulian theorem again) 
$(y_n)$ contains a \text{\rm w}-convergent subsequence, as desired.
\end{proof}

\subsection{d-Limited operators.}
Let us have a look at limited sets.

\begin{definition}\label{d-limited-def}
{\em 
An operator $T\in\text{\rm L}(E,Y)$ is {\em\text{\rm d}-limited} 
({\em $\sigma$\text{\rm d}-limited}),
if $T(D)$ is limited for every \text{\rm d-bdd} (countable) subset $D$ of $E$.
The set of such operators is denoted by $\text{\rm d-Lim}(E,Y)$
(resp., by $\sigma\text{\rm d-Lim}(E,Y)$).}
\end{definition}

\noindent
Clearly, limited operators and \text{\rm d}-compact operators are \text{\rm d}-limited. 
The following example shows that a \text{\rm d}-limited operator need not be limited.

\begin{example}\label{example-d-L0}{\em
Let $T:L^\infty[0,1]\to L^1[0,1]$ be the natural embedding. By Example \ref{example-d-c0},
$T$ is \text{\rm d}-compact, and hence $T\in\text{\rm d-Lim}(L^\infty[0,1],L^1[0,1])$.
The limited sets agree with compact sets in $L^1[0,1]$
by its separability. So, $T\notin\text{\rm Lim}(L^\infty[0,1],L^1[0,1])$
since $T$ is not compact.}
\end{example}

\noindent
By the next example, a limited operator 
need not be \text{\rm d}-compact.

\begin{example}\label{example-d-L1}
{\em
Let $T:c_0\to \ell^\infty$ be the embedding operator.
By Phillip's lemma (cf., \cite[Theorem 4.67]{AB1}), $T(B_{c_0})$
is a limited subset of $\ell^\infty$. Therefore, $T\in\text{\rm Lim}(c_0,\ell^\infty)$.
However, $T(\{e_n: n\in\mathbb{N}\})=\{e_n: n\in\mathbb{N}\}$
is not relatively compact in $\ell^\infty$, and hence
$T\notin\text{\rm d-K}(c_0,\ell^\infty)$.}
\end{example}

\noindent
In general, \text{\rm w}-compact operators need not to be \text{\rm d}-limited.

\begin{example}\label{example-d-L4}
{\em
Since $\ell^2$ is reflexive and $\{e_n: n\in\mathbb{N}\}$ is a non-limited \text{\rm d-bdd} 
subset of $\ell^2$ then $I_{\ell^2}\in\text{\rm W}(\ell^2)\setminus\text{\rm d-Lim}(\ell^2)$.
}
\end{example}

\noindent
We have no example of an operator $T\in\text{\rm d-Lim}(E,Y)\setminus\text{\rm W}(E,Y)$.

\medskip
The following fact is analogues to Propositions \ref{d-K-elem-1} and \ref{d-W-elem-1}.

\begin{proposition}\label{d-L-elem-1}
{\em
\begin{enumerate}[$i)$]
Let $T\in\text{\rm L}(E,Y)$. The following conditions are equivalent.
\item
$T\in\text{\rm d-Lim}(E,Y)$.
\item
$T(D)$ is limited for every disjoint normalized $D\subseteq E_+$.
\item
$T\in\sigma\text{\rm d-Lim}(E,Y)$.
\item
$T(D)$ is limited for every disjoint normalized countable $D\subseteq E_+$.
\end{enumerate}
}
\end{proposition}

\begin{proof}
Implications $i)\Longrightarrow ii)$ and $i)\Longrightarrow iii)\Longrightarrow iv)$ are trivial.
The proof of $ii)\Longrightarrow i)$ is similar to the proof of $ii)\Longrightarrow i)$
in Proposition \ref{d-K-elem-1} by using the limitedness of the set 
$\text{\rm cch}(A_1)+\text{\rm cch}(A_2)$ 
for limited subsets $A_1$ and $A_2$ of $Y$ (see \cite{BD}).

$iv)\Longrightarrow ii)$\
Let $D$ be a disjoint normalized subset of $E_+$. 
On the way to a contradiction, assume that $T(D)$ is not limited. 
So, there exists a \text{\rm w}$^\ast$-null sequence $(f_n)$ in $Y'$ such that 
$f_n\not\rightrightarrows 0(T(D))$. Then, for every $n\in\mathbb N$,
there exists an element $d_n\in D$ such that for the resulting sequence $|f_n(Td_n)|\not\rightarrow 0$. 
WLOG, suppose all $d_n$ to be distinct (and therefore disjoint).
So, there exist an $\epsilon>0$ and subsequence $(n_k)_k$ such that  
$|f_{n_k}(Td_{n_k})|\geq\epsilon$ for all $k$, which contradicts $iv)$.
\end{proof}

The ``almost version'' of Definition \ref{d-limited-def} and the corresponding 
modifications of Proposition \ref{d-L-elem-1} for the space $\text{\rm d-aLim}(E,F)$ 
are left to the reader as exercises.

\subsection{d-Andrews operators.}
The last kind of operators that is considered in the present paper 
relies on the Dunford--Pettis sets.

\begin{definition}\label{d-DP-def}{\em 
An operator $T\in\text{\rm L}(E,Y)$ is a {\em\text{\rm d}-Andrews} 
({\em $\sigma$\text{\rm d}-Andrews}) operator,
if $T(D)$ is a Dunford--Pettis set for every \text{\rm d-bdd} (countable) subset $D$ of $E$.
The set of such operators is denoted by $\text{\rm d-And}(E,Y)$
(resp., by $\sigma\text{\rm d-And}(E,Y)$).}
\end{definition}

\noindent
Clearly, $\text{\rm d-K}(E,Y)\subseteq\text{\rm d-Lim}(E,Y)\subseteq\text{\rm d-And}(E,Y)$. 
As mentioned in Remark \ref{Andrews operators}~c), 
$I_{c_0}\in\text{\rm And}(c_0)\setminus\text{\rm Lim}(c_0)$.
The operator $I_{c_0}$ is even not \text{\rm d}-limited. Indeed, $(c_0)'=\ell^1\ni e_n\convws 0$
yet $e_n\not\rightrightarrows 0(D)$ for the disjoint subset $D=\{e_n:n\in\mathbb{N}\}$ of $B_{c_0}$
because of $e_n(e_n)\equiv 1$. For another example of an Andrews operator which is not
\text{\rm d}-limited, see Example \ref{example-d-LA1}. 
The next example shows that a \text{\rm d}-Andrews operator need not be  Andrews operator.

\begin{example}\label{example-d-And0}{\em
Let $1<p<\infty$ and let $T:L^\infty[0,1]\to L^p[0,1]$ be the natural embedding. 
It follows from Remark \ref{Andrews operators}, that 
$\text{\rm And}(L^\infty[0,1],L^p[0,1])=\text{\rm K}(L^\infty[0,1],L^p[0,1])$.
By Example \ref{example-d-c0}, $T\notin\text{\rm K}(L^\infty[0,1],L^p[0,1])$,
and hence $T\notin\text{\rm And}(L^\infty[0,1],L^p[0,1])$. 
Example \ref{example-d-c0} implies $T\in\text{\rm d-K}(L^\infty[0,1],L^p[0,1])$,
and hence $T\in\text{\rm d-And}(L^\infty[0,1],L^p[0,1])$.}
\end{example}

The proof of the next fact is omitted as it is similar to the proof of Proposition~\ref{d-L-elem-1}.

\begin{proposition}\label{d-A-elem-1}
{\em
\begin{enumerate}[$i)$]
Let $T\in\text{\rm L}(E,Y)$. The following conditions are equivalent.
\item
$T\in\text{\rm d-And}(E,Y)$.
\item
$T(D)$ is a Dunford--Pettis set for every disjoint normalized $D\subseteq E_+$.
\item
$T\in\sigma\text{\rm d-And}(E,Y)$.
\item
$T(D)$ is a Dunford--Pettis set for every disjoint normalized countable $D\subseteq E_+$.
\end{enumerate}
}
\end{proposition}

We left ``almost versions'' of Definition \ref{d-DP-def} and Proposition \ref{d-A-elem-1} 
for $\text{\rm d-aAnd}(E,F)$ as exercises.

\section{d-(Weakly) compact operators}

\hspace{5mm}
Here, we investigate \text{\rm d}-(weakly) compact operators.
The following Grothendieck type lemma
(cf. \cite[Theorem 3.44]{AB1} and \cite[Proposition 1]{Schlum}) is useful.

\begin{lemma}\label{Grothendieck type lemma}{\em
Let, for every $\varepsilon > 0$, there exists an $A_\varepsilon \subseteq X$ such that 
$A\subseteq A_\varepsilon + \varepsilon B_X$.
\begin{enumerate}[$i)$]
\item 
If each $A_\varepsilon$ is relatively compact then $A$ is relatively compact.
\item 
If each $A_\varepsilon$ is relatively weakly compact then $A$ is relatively weakly compact.
\item 
If each $A_\varepsilon$ is (almost) limited then $A$ is (almost) limited.
\item 
If each $A_\varepsilon$ is an (almost) Dunford--Pettis set then $A$ is an (almost) Dunford--Pettis set.
\end{enumerate}}
\end{lemma}

\begin{proof}
We include only a proof of $iv)$. 
Assume, for every $\varepsilon>0$, there exists a Dunford--Pettis set $A_\varepsilon$
with $A\subseteq A_\varepsilon+\varepsilon B_X$, and let $(f_n)$ be \text{\rm w}-null
in $X'$.
Then $\|f_n\|\le M$ for some $M\in\mathbb{R}$ and all $n$. 
Since $\varepsilon>0$ is arbitrary, the inequality
$$
   \limsup_{n\to\infty}\bigl(\sup_{x\in A}|f_n(x)|\bigl)\le
   \limsup_{n\to\infty}\bigl(\sup_{y\in A_\varepsilon}|f_n(y)|\bigl)+
   \limsup_{n\to\infty}\bigl(\sup_{z\in\varepsilon B_X}|f_n(z)|\bigl)\le\varepsilon M
$$
implies $f_n\rightrightarrows 0(A)$, and hence $A$ is a Dunford--Pettis set.

For the almost Dunford--Pettis case, letting $X$ to be a Banach lattice, 
just replace in the proof above  \text{\rm w}-null $(f_n)$ by 
 disjoint \text{\rm w}-null $(f_n)$.
\end{proof}

\subsection{Norm closedness in spaces of d-(weakly) compact operators.}
It is straightforward to see that $\text{\rm d-K}(E,Y)$ and $\text{\rm d-W}(E,Y)$
are vector spaces. An application of Lemma \ref{Grothendieck type lemma} 
gives the norm closedness of these spaces. 

\begin{theorem}\label{d-K-closed}{\em
Suppose $\|T_n-T\|\to 0$.
\begin{enumerate}[$i)$]
\item 
If $T_n\in\text{\rm d-K}(E,Y)$ for all $n$ then $T\in\text{\rm d-K}(E,Y)$.
\item 
If $T_n\in\text{\rm d-W}(E,Y)$ for all $n$ then $T\in\text{\rm d-W}(E,Y)$. 
\end{enumerate}}
\end{theorem}

\begin{proof}
Let $T_n\in\text{\rm d-K}(E,Y)$ (resp., $T_n\in\text{\rm d-W}(E,Y)$) 
for all $n$. Let $D$ be a disjoint subset of $B_E$. 
Take an arbitrary $\varepsilon>0$ and pick $k\in\mathbb{N}$ with $\|T-T_k\|\le\varepsilon$.
Since $T_k\in\text{\rm d-K}(E,Y)$ (resp., $T_k\in\text{\rm d-W}(E,Y)$) 
then $T_k(D)$ is a relatively compact (resp., \text{\rm w}-compact) subset of $Y$.
As $T(D)\subseteq T_k(D)+\varepsilon B_Y$,
Lemma \ref{Grothendieck type lemma} implies that $T(D)$ is a 
relatively compact (resp., \text{\rm w}-compact) subset of $Y$. 
\end{proof}
\noindent
In view of Proposition \ref{enveloping norm}, we obtain:

\begin{corollary}\label{r-d-K}{\em
For every $E$ and $F$, $\text{\rm span}\bigl(\text{\rm d-K}_+(E,F)\bigl)$ and
$\text{\rm span}\bigl(\text{\rm d-W}_+(E,F)\bigl)$ are Banach spaces under 
their enveloping norms.}
\end{corollary}

\subsection{When is each bounded operator d-(weakly) compact.}
In view of Proposition \ref{d-K-elem-1}, each bounded operator $T:E\to E$ is \text{\rm d}-compact
iff $I_E\in\text{\rm d-K}(E)$ iff every disjoint normalized $D\subseteq E_+$ is finite iff
$\dim(E)<\infty$.
It is well known that every bounded $T:X\to Y$ is weakly compact, whenever
one of the Banach spaces $X$ and $Y$ is reflexive (cf., \cite[Theorem 5.24]{AB1}).
Here, we give a condition on $E$ under which every bounded operator
$T:E\to Y$ is \text{\rm d}-weakly compact.

\begin{theorem}\label{KB-d-W}{\em
The following conditions are equivalent.
\begin{enumerate}[$i)$]
\item 
$\text{\rm L}(E,Y)=\text{\rm d-W}(E,Y)$ for every $Y$.
\item 
$\text{\rm L}(E)=\text{\rm d-W}(E)$.  
\item 
$E'$ is a KB-space.
\end{enumerate}}
\end{theorem}

\begin{proof}
$i)\Longrightarrow ii)$\ is trivial.

$ii)\Longrightarrow iii)$\ 
Let $(x_n)$ be a disjoint bounded sequence in $E$. 
By the condition $ii)$, $I_E\in\text{\rm d-W}(E)$, and hence
the set $W=\{x_n\}_{n=1}^\infty$ is relatively \text{\rm w}-compact.
It follows from \cite[Theorem 4.34]{AB1} that every disjoint sequence in the solid hull 
of $W$ is \text{\rm w}-null, so $(x_n)$ is \text{\rm w}-null.
Now, $E'$ is a KB-space by \cite[Theorem 2.4.14]{Mey}.

$iii)\Longrightarrow i)$\
Let $T\in\text{\rm L}(E,Y)$ and $D$ be a \text{\rm d-bdd} subset of $E$.
Take a sequence $(y_n)$ in $T(D)$. If $(y_n)$ has only finitely many distinct elements
then $(y_n)$ obviously contains even a norm-convergent subsequence.
If $(y_n)$ has infinitely many distinct elements then there is a distinct (and hence disjoint)
sequence $(x_{n_k})$ of elements of $D$ such that $y_{n_k}=Tx_{n_k}$.
Since $E'$ is a KB-space then $x_{n_k}\convw 0$ by \cite[Theorem 2.4.14]{Mey},
and hence $y_{n_k}=Tx_{n_k}\convw 0$. Therefore, $T(D)$ is relatively \text{\rm w}-compact,
and consequently $T\in\text{\rm d-W}(E,Y)$.
\end{proof}

\subsection{Domination problem for d-(weakly) compact operators.}
It is well known that the (weak) compactness does not pass to dominated operators 
in general (cf., \cite[Sections 5.1 and 5.2]{AB1}). The following example shows that,
in the classes of \text{\rm d}-(weakly) compact operators, the situation is not better.

\begin{example}\label{example-d-W5}
{\em
Consider operators $S,T:\ell^1\to c$ defined by
$$
    Sx=\sum\limits_{k=1}^{\infty}\Bigl(\sum\limits_{i=k}^{\infty}x_i\Bigl)e_k \ \ \ \& \ \ \
    Tx=\sum\limits_{k=1}^{\infty}\Bigl(\sum\limits_{i=1}^{\infty}x_i\Bigl)e_k.
$$
Then $0\le S\le T$ and $T$ is compact (even rank one), 
yet $S\notin\text{\rm d-W}(\ell^1,c)$ as the sequence
$(Se_n)=\Bigl(\sum\limits_{k=1}^ne_k\Bigl)$ in $c$ has no \text{\rm w}-convergent subsequence. 
}
\end{example}

\noindent
By the Wickstead theorem, (cf., \cite[Theorem 5.31]{AB1}), 
if either $E'$ or $F$ has order continuous norm then 
$0\le S\le T\in\text{\rm W}(E,F)$ implies $S\in\text{\rm W}(E,F)$.
The following theorem deal with the domination problem 
for d-(weakly) compact operators.

\begin{theorem}\label{KB-d-K-W-domin}{\em
Let $0\le S\le T\in\text{\rm L}(E,F)$, where $E'$ is a KB-space. 
\begin{enumerate}[$i)$]
\item
If $T\in\text{\rm d-K}(E,F)$ then $S\in\text{\rm d-K}(E,F)$.
\item
If $T\in\text{\rm d-W}(E,F)$ then $S\in\text{\rm d-W}(E,F)$.
\end{enumerate}}
\end{theorem}

\begin{proof}
$i)$\ 
Let $D$ be a disjoint normalized countable subset of $E_+$.
Accordingly to Proposition \ref{d-K-elem-1}, we need to show that $S(D)$
is relatively compact. So, let $(y_n)$ be a sequence in $S(D)$.
If $y_n=0$ for infinitely many $n$-th then $(y_n)$ has a norm-null subsequence.
If $y_n=0$ only for finitely many $n$-th, by removing them, we may suppose 
all $y_n\ne 0$. For each $n$, pick $x_n\in D$ 
with $y_n=Sx_n$ and denote $D_0=\bigl\{x_n: n\in\mathbb{N}\bigl\}$. 
Then $D_0$ is a disjoint normalized subset of $E_+$.
If $D_0$ is finite then $S(D_0)$ is finite and hence compact. So, suppose $D_0$ is infinite. 
By the assumption $i)$, $T(D_0)$ is relatively compact. Since $E'$ is a KB-space,
the disjoint normalized sequence $(x_n)$ is \text{\rm w}-null, and hence $Tx_n\convw 0$. 
Suppose $\|Tx_n\|\not\to 0$. Then $\|Tx_{n_k}\|\ge M>0$ for some subsequence
$(x_{n_k})$ of $(x_n)$. Since $T\in\text{\rm d-K}(E,F)$, there exists a further
subsequence $(x_{n_{k_j}})$ such that $\|Tx_{n_{k_j}}-z\|\to 0$, and hence
$Tx_{n_{k_j}}\convw z$. Since $Tx_n\convw 0$ then $z=0$, and hence $\|Tx_{n_{k_j}}\|\to 0$ violating 
$\|Tx_{n_k}\|\ge M>0$ for all $k$. The obtained contradiction proved $\|Tx_n\|\to 0$.
Using $0\le S\le T$, we obtain $\|y_n\|=\|Sx_n\|\to 0$. Thus $S\in\text{\rm d-K}(E,F)$.

$ii)$\ 
It follows from Theorem \ref{KB-d-W}.
\end{proof}

\noindent
We do not know whether or not the condition on $E'$ in Theorem \ref{KB-d-K-W-domin} 
is necessary for conclusions in $i)$ and/or in $ii)$.

\section{d-Limited and d-Andrews operators}

\hspace{5mm}
Here, we present some properties of \text{\rm d}-limited and \text{\rm d}-Andrews operators.
In the end of the section, we discuss their ``almost versions''.

\subsection{Norm closedness of the space of d-limited and d-Andrews operators.}
It is easy to see that $\text{\rm d-Lim}(E,Y)$ and $\text{\rm d-And}(E,Y)$
are vector spaces. We omit the proof of the following theorem which is
similar to the proof of Theorem \ref{d-K-closed}.

\begin{theorem}\label{d-Lim-closed}{\em
Suppose $\|T_n-T\|\to 0$.
\begin{enumerate}[$i)$]
\item 
If $T_n\in\text{\rm d-Lim}(E,Y)$ for all $n$ then $T\in\text{\rm d-Lim}(E,Y)$.
\item 
If $T_n\in\text{\rm d-And}(E,Y)$ for all $n$ then $T\in\text{\rm d-And}(E,Y)$. 
\end{enumerate}}
\end{theorem}

\noindent
By using Proposition \ref{enveloping norm}, we obtain:

\begin{corollary}\label{r-d-Lim}{\em
For every $E$ and $F$, $\text{\rm span}\bigl(\text{\rm d-Lim}_+(E,F)\bigl)$ and
$\text{\rm span}\bigl(\text{\rm d-And}_+(E,F)\bigl)$ are Banach spaces under 
their enveloping norms.}
\end{corollary}

\subsection{Domination problem for d-limited and d-Andrews operators.}
First, we show that in general the domination problem for d-limited operators has a 
negative solution. Then, we give natural conditions on the range under providing 
the positive solutions for d-limited and d-Andrews operators. 

\begin{example}\label{example-d-LA1}
{\em
Example \ref{example-d-W5} presents a non d-limited operator
 that is dominated by a rank one operator.
To see this, consider the disjoint \text{\rm w}$^\ast$-null sequence $(e_{2n}-e_{2n-1})$ 
in $c'=\ell^1$. For a \text{\rm d-bdd} subset $D=\{e_{2n-1}:n\in\mathbb{N}\}$ of $\ell^1$,
$$
   \limsup\limits_{n\to\infty}\bigl(\sup\limits_{d\in D}|(e_{2n}-e_{2n-1})(Sd)|\bigl)\ge 
   \limsup\limits_{n\to\infty}|(e_{2n}-e_{2n-1})(Se_{2n-1})|=1,
$$
and hence $e_{2n}-e_{2n-1}\not\rightrightarrows 0(S(D))$. Therefore $S\notin\text{\rm d-Lim}(\ell^1,c)$. However, $S$ is an Andrews operator since  bounded subsets of $c$ are Dunford--Pettis sets.
}
\end{example}

\begin{theorem}\label{ws-seq-Lim-dom}{\em
Let $0\le S\le T\in\text{\rm d-Lim}(E,F)$. If $F'$ has sequentially \text{\rm w}$^\ast$-continuous 
lattice operations then $S\in\text{\rm d-Lim}(E,F)$.}
\end{theorem}

\begin{proof}
Let $D$ be a disjoint normalized subset of $E_+$.
Let $f_n\convws 0$ in $F'$. As lattice operations in $F'$ are sequentially 
\text{\rm w}$^\ast$-continuous then $|f_n|\convws 0$. Since $T$ is \text{\rm d}-limited
then $|f_n|\rightrightarrows 0(T(D))$ by Proposition \ref{d-L-elem-1}. 
Since $0\le|f_n(Sd)|\le |f_n|(Sd)\le |f_n|(Td)$ 
for every $d\in D$ then $f_n\rightrightarrows 0(S(D))$, and hence (by Proposition \ref{d-L-elem-1}) $S\in\text{\rm d-Lim}(E,F)$.
\end{proof}

\begin{theorem}\label{w-seq-And-dom}{\em
Let $0\le S\le T\in\text{\rm d-And}(E,F)$. If $F'$ has sequentially \text{\rm w}-continuous
lattice operations then $S\in\text{\rm d-And}(E,F)$.}
\end{theorem}

\begin{proof}
Let $D$ be a disjoint normalized subset of $E_+$.
Let $f_n\convw 0$ in $F'$. Then $|f_n|\convw 0$. Since $T\in\text{\rm d-And}(E,F)$ 
then $|f_n|\rightrightarrows 0(T(D))$.
Since $0\le|f_n(Sd)|\le |f_n|(Sd)\le |f_n|(Td)$ for every $d\in D$ then
$f_n\rightrightarrows 0(S(D))$, and hence $S\in\text{\rm d-And}(E,F)$.
\end{proof}

%
%
%

\subsection{Some remarks on d-a-limited and d-a-Andrews operators.}
Straightforward modification of Theorem \ref{d-Lim-closed} and  
Corollary \ref{r-d-Lim} is left as an exercise to the reader. 
The ``almost versions'' of Theorems \ref{ws-seq-Lim-dom} and \ref{w-seq-And-dom}
are slightly more interesting. Recall that $F$ has property \text{\rm (d)} if
$|f_n|\convws 0$ for every disjoint \text{\rm w}$^\ast$-null 
sequence $(f_n)$ in $F'$ \cite[Definition 1]{Elbour}.

\begin{theorem}\label{d-aLim-aAnd-dom}{\em
Let $0\le S\le T\in\text{\rm L}(E,F)$.
\begin{enumerate}[$i)$]
\item 
If $T\in\text{\rm d-aLim}(E,F)$ and $F$ has the property \text{\rm (d)} then
$S\in\text{\rm d-aLim}(E,F)$.
\item 
If $T\in\text{\rm d-aAnd}(E,F)$ then $S\in\text{\rm d-aAnd}(E,F)$. 
\end{enumerate}}
\end{theorem}

\begin{proof}
$i)$\ 
Let $D$ be a \text{\rm d-bdd} subset of $E$.
Then $|D|=\{|d|: d\in D\}$ is also a \text{\rm d-bdd} subset of $E$.
Let $(f_n)$ be a disjoint \text{\rm w}$^\ast$-null sequence in $F'$. 
By the property \text{\rm (d)},
$|f_n|\convws 0$. Since $T\in\text{\rm d-aLim}(E,F)$
then $|f_n|\rightrightarrows 0(T(|D|))$. Since 
$$
   0\le|f_n(Sd)|\le|f_n|(S|d|)\le|f_n|(T|d|) \ \ \ \ \ (\forall d\in D),
   \eqno(4)
$$ 
then $f_n\rightrightarrows 0(S(D))$, and $S\in\text{\rm d-aLim}(E,F)$.

$ii)$\ 
Let $D$ be a \text{\rm d-bdd} subset of $E$
and let $(f_n)$ be a disjoint \text{\rm w}-null sequence in $F'$. 
Then $|f_n|\convw 0$ by \cite[Theorem 4.34]{AB1}. 
Since $T\in\text{\rm d-aAnd}(E,F)$ and since $|D|$ 
is a \text{\rm d-bdd} subset of $E$ then $|f_n|\rightrightarrows 0(T(|D|))$. 
Formula (4) implies $f_n\rightrightarrows 0(S(D))$, and consequently $S\in\text{\rm d-aAnd}(E,F)$.
\end{proof}
\noindent
The condition that $F$ has the property \text{\rm (d)} is essential 
in Theorem \ref{d-aLim-aAnd-dom}~$i)$. Indeed, the range space $c$ in 
Example \ref{example-d-LA1} does not have the property \text{\rm (d)} as the disjoint sequence
$(e_{2n+1}-e_{2n})$ is \text{\rm w}$^\ast$-null in $\ell^1=c'$ yet
$|e_{2n+1}-e_{2n}|\bigl(\sum_{k=1}^\infty e_k\bigl)\equiv 2$. \\

\bibliographystyle{plain}

\addcontentsline{toc}{section}{KAYNAKLAR}
\begin{thebibliography}{99}
\normalsize	


\bibitem{AA}
Abramovich, Y. A, Aliprantis, C. D.,
{\em An invitation to operator theory.}
American Math. Soc., Providence, RI, (2002).

\bibitem{AB1}
Aliprantis, C. D., Burkinshaw, O., 
{\em Positive operators.}
Springer, Dordrecht, (2006).

\bibitem{AEG0}
Alpay, S., Emelyanov, E., Gorokhova, S., 
{\em Enveloping norms of regularly P-operators in Banach lattices.}
Positivity 28, 37 (2024).

\bibitem{AEG}
Alpay, S., Emelyanov, E., Gorokhova, S., 
{\em Duality and norm completeness in the classes of limitedly Lwc and Dunford--Pettis Lwc operators.}
Turkish Math. J. 48\,(2), 267--278 (2024).

\bibitem{Andr}
Andrews, K. T.,
{\em Dunford--Pettis sets in the space of Bochner integrable functions.}
Math. Ann. 241, 35--41 (1979).

\bibitem{BD}
Bourgain, J., Diestel, J., 
{\em Limited operators and strict cosingularity.}
Math. Nachr. 119, 55--58 (1984).

\bibitem{Bour}
Bouras, K.,
{\em Almost Dunford–Pettis sets in Banach lattices.}
Rend. Circ. Mat. Palermo 62, 227--236 (2013).

\bibitem{CCJ}
Chen, J. X., Chen, Z. L., Ji, G. X., 
{\em Almost limited sets in Banach lattices.} 
J. Math. Anal. Appl. 412, 547--553 (2014).

\bibitem{Elbour}
Elbour, A., 
{\em Some characterizations of almost limited operators.}
Positivity 21, 865--874 (2017).

\bibitem{Emel}
Emelyanov, E.,  
{\em Collective order convergence and
collectively qualified set of operators.}
https://arxiv.org/pdf/2408.03671v4

\bibitem{Gor}
Gorokhova, S.,  
{\em On compact (limited) operators between Hilbert and Banach spaces.}
https://arxiv.org/abs/2403.06262v5

\bibitem{Mey}
Meyer-Nieberg, P., 
{\em Banach lattices.}
Springer-Verlag, Berlin, (1991).

\bibitem{Schlum}
Schlumprecht, T., 
{\em  Limited Sets in \text{\rm C(K)}-Spaces and Examples Concerning the Gelfand-Phillips Property.}
Math. Nachr. 157, 51--64, (1992).

\end{thebibliography}
\end{document}